\documentclass{amsart}

\usepackage{amsmath}    
\usepackage{graphicx}   
\usepackage{verbatim}   
\usepackage{amssymb}
\usepackage{amsthm}
\usepackage{mathrsfs}
\usepackage{bbm}
\usepackage{exscale}
\usepackage{enumerate}
\usepackage[noadjust]{cite}
\usepackage{color}
\newtheorem{theorem}{Theorem}[section]
\newtheorem{lemma}[theorem]{Lemma}
\newtheorem{corollary}[theorem]{Corollary}
\newtheorem{proposition}[theorem]{Proposition}

\theoremstyle{remark}
\newtheorem{rem}[theorem]{Remark}
\newcommand{\finsum}[3]{
\underset{#1=#2}{\overset{#3}\sum}}

\newcommand{\dint}{
\displaystyle\int}

\newcommand{\dsum}{
\displaystyle\sum}

\newcommand{\inflim}[1]{
\underset{#1\to\infty}\lim}

\newcommand{\zsum}[1]{
\underset{#1\in\Z}\dsum}

\newcommand{\zsumd}[1]{
\underset{#1\in\Z^d}\dsum}

\newcommand{\R}{
\mathbb{R}}
\newcommand{\C}{
\mathbb{C}}
\newcommand{\N}{
\mathbb{N}}
\newcommand{\Z}{
\mathbb{Z}}

\newcommand{\supp}{
\textnormal{supp}}
\newcommand{\T}{
\mathbb{T}}

\newcommand{\schwartz}{
\mathscr{S}}
\newcommand{\bracket}[1]{
\left\langle#1\right\rangle}

\newcommand{\phica}{
\widehat{\phi_{\alpha,c}}}

\newcommand{\nutilde}{
\widetilde{\nu}}
\newcommand{\wtilde}{
\widetilde{\omega}}

\newcommand{\Lachat}{
\widehat{L_{\alpha,c}}}

 \newcommand{\Lh}{
 L_{\alpha,\frac{1}{h}}}
 \newcommand{\Lhhat}{
 \widehat{\Lh}}
 \newcommand{\chih}{
 \chi_{\alpha,h}}
 \newcommand{\chihhat}{
 \widehat{\chih}}

 \newcommand{\eps}{
 \varepsilon}
 \newcommand{\Mp}{
 \mathcal{M}_p}
 
 \newcommand{\floor}[1]{
 \lfloor#1\rfloor}
 \newcommand{\ceiling}[1]{
 \lceil#1\rceil}

\newcommand{\mah}{
m_{\alpha,h}}

\newcommand{\multi}[1]{[#1]}

 \title[Multiquadric Interpolation: Convergence Rates]{Cardinal Interpolation With General Multiquadrics: Convergence Rates}
 
\author{Keaton Hamm}
\address{Department of Mathematics, Vanderbilt University, Nashville, TN, 37212}
\email{keaton.hamm@vanderbilt.edu} 

\author{Jeff Ledford}
\address{Department of Mathematics and Applied Mathematics, Virginia Commonwealth University, Richmond, VA, 23284}
\email{jpledford@vcu.edu}


 \thanks{Part of the work for this article was completed when the first author was a graduate student at Texas A\&M University, where he was partially supported by National Science Foundation grants DMS 1160633 and 1464713.  The second author was partially supported by the Workshop in Analysis and Probability at Texas A\&M University.}

\begin{document}
\begin{abstract}

This article pertains to interpolation of Sobolev functions at shrinking lattices $h\Z^d$ from $L_p$ shift-invariant spaces associated with cardinal functions related to general multiquadrics, $\phi_{\alpha,c}(x):=(|x|^2+c^2)^\alpha$.  The relation between the shift-invariant spaces generated by the cardinal functions and those generated by the multiquadrics themselves is considered.  Additionally, $L_p$ error estimates in terms of the dilation $h$ are considered for the associated cardinal interpolation scheme.  This analysis expands the range of $\alpha$ values which were previously known to give such convergence rates (i.e. $O(h^k)$ for functions with derivatives of order up to $k$ in $L_p$, $1<p<\infty$).  Additionally, the analysis here demonstrates that some known best approximation rates for multiquadric approximation are obtained by their cardinal interpolants.
\end{abstract}
 
 \maketitle
\allowdisplaybreaks

\section{Introduction}

This article is primarily concerned with cardinal interpolation schemes associated with {\em general multiquadrics} in higher dimensions defined via two parameters as $\phi_{\alpha,c}(x):=(|x|^2+c^2)^\alpha$.  The two main objects of study are the principal shift-invariant spaces associated with either the multiquadric {\em cardinal functions} or the multiquadrics themselves, and the performance of the interpolants from these spaces for recovery of functions in the classical $L_p$ Sobolev spaces of finite smoothness.  Cardinal interpolation finds its origins in the work of Schoenberg (\cite{Schoenberg} and references therein), who studied interpolation at the integer lattice using splines.  Subsequent investigations ensued involving cardinal functions associated with radial basis functions (RBFs), including much work by Buhmann \cite{BuhmannOdd, Buhmann, BuhmannBook}, Baxter, Riemenschneider, and Sivakumar \cite{Baxter, RiemSiva,rs1,rs2, rs3, siva}, and the authors \cite{HammLedford}.   Some of the RBFs considered in those works are the thin plate spline, the Gaussian kernel, and the Hardy multiquadric.  Recently, \cite{Ledford} provided a general framework for recovering bandlimited functions from their samples at the (multi) integer lattice using so-called {\em regular families of cardinal interpolators}, of which certain families of Gaussians and multiquadrics are examples.

This article is a continuation of the study done in \cite{HammLedford}, in which cardinal interpolation in one dimension using  general multiquadrics was considered.  There, detailed estimates on the univariate cardinal functions were given, the behavior of the interpolation operators acting on $\ell_p$ spaces was considered, and a method for recovery of multivariate bandlimited functions from their multiquadric interpolants via a limiting process was shown.  One interesting problem for these interpolation schemes is to determine how quickly the interpolant of a function converges (globally) to the function based on its smoothness.   

Much work has been done on determining convergence rates for functions in the so-called {\em native space} of a given RBF (which for positive definite RBFs is the reproducing kernel Hilbert space with the RBF as the kernel).  When interpolating functions in the native space, convergence is often exponentially fast \cite{MadychNelson,Wendland}, however this space is often rather small.  Indeed, for the Gaussian kernel, the native space consists of functions whose Fourier transform satisfies $\widehat{f}e^{|\cdot|^2}\in L_2(\R^d)$ \cite[Theorem 10.12]{Wendland}. Consequently, it is desirable to determine the rate of approximation for more general classes of smooth functions.  Here, we consider interpolation of Sobolev ($W_p^k(\R^d)$) functions.  The inspiration for our work is the article of Hangelbroek, Madych, Narcowich, and Ward \cite{hmnw}, which provided convergence rates for Gaussian interpolation.

Often in the RBF literature, interpolants take the form $\sum_{j\in\Z^d}c_j\phi(\cdot-j)$, where $\phi$ is the given RBF.  However, associated with many RBFs are cardinal functions $L_\phi$ satisfying $L_\phi(k)=\delta_{0,k}$, $k\in\Z^d$, and so another interpolant is $\sum_{j\in\Z^d}a_jL_\phi(\cdot-j)$ where $a_j=f(j)$ for a given function $f$.  This brings up some natural questions of how the series above converge, and in what sense the interpolants are related.  In particular, for Sobolev functions to be considered here, the coefficients typically lie in $\ell_p(\Z^d)$, and so a natural object of study are the $L_p$ principal shift-invariant spaces associated with $\phi$ and $L_\phi$ (see Section \ref{SECBasic} for the precise definition of these spaces).  In many instances, it is easily shown that the shift-invariant spaces coincide (regardless of whether one defines them in terms of $\phi$ or $L_\phi$); however, for growing kernels, e.g. multiquadrics with positive $\alpha$, the matter is a bit more delicate since one of the spaces is not well-defined.  In many instances, one still has that $L_\phi = \sum_{j\in\Z^d}d_j\phi(\cdot-j)$ as an absolutely convergent series.

The advent of shift-invariant space techniques to the approximation problem by de Boor, DeVore, and Ron \cite{deBoorRon,DDR} was important because (among many other things) it gave rise to optimal approximation rates for classes of functions which were much broader than the aforementioned native space of the RBF.  The power of these methods sparked a slew of ideas and further generalizations, including the work of Buhmann and Ron \cite{BuhmannRon}, Jia \cite{jia}, Johnson \cite{Johnson,Johnson2,Johnson3}, and Kyriazis \cite{K}.  Many of these references consider the best rates of approximation of smooth functions from shift-invariant spaces associated with different RBFs.  Our study here demonstrates that in most cases, the optimal approximation rates using multiquadrics can be achieved by the associated cardinal interpolants.

For more information on stable computation of multiquadric approximants, the interested reader is invited to consult the works of Driscoll, Fornberg, and Flyer (\cite{Fornberg,Fornberg2,FornbergFlyer} and references therein).  Additionally, for a prolonged discussion of other methods in use with many references, the reader may consult \cite{HammLedford}.  It should also be noted that \cite{BD} gives spectral approximation orders for multiquadric interpolation at $h\Z^n$ on compact domains for functions whose Fourier transforms satisfy a certain decay property.

The rest of the paper is laid out as follows.  In Section \ref{SECBasic}, we provide some preliminaries such as notation and facts about the general multiquadrics and shift-invariant spaces.  Section \ref{SECMain} details the statements of our main results and mentions some of the key ingredients; this section also contains the proof of the main theorem on approximation rates for Sobolev interpolation at the shrinking lattice $h\Z^d$ (Theorem \ref{THMmaintheorem}).   Section \ref{SECMultiplier} begins by discussing the Fourier multiplier associated with the multiquadric cardinal function, whose operator norm governs the rest of the analysis.  Pointwise and norm estimates are given for the multiplier operator as well.  Section \ref{SECProofs} contains the lion's share of the proofs of the main theorems of Section \ref{SECMain}, while the Appendix contains the distributional proof of one of the driving equations.  We conclude with some brief remarks and extensions in Section \ref{SECremark}.


\section{Preliminaries}\label{SECBasic}

Let $\Omega\subset\R^d$ be an open set. Then let $L_p(\Omega)$, $1\leq p\leq\infty$, be the usual space of $p$--integrable functions on $\Omega$ with its usual norm.  If no set is specified, we mean $L_p(\R^d)$.  Similarly, denote by $\ell_p$ the usual sequence spaces indexed by the (multi) integers. Throughout, $\gamma$ and $\beta$ will be multi-indices, with $D^\gamma$ taking on its usual meaning as the differential operator.  To mitigate confusion, the convention $[\gamma]:=\sum_{j=1}^d\gamma_j$ will be used to denote the length of the multi-index since the symbol $|\cdot|$ is reserved exclusively for the Euclidean distance on $\R^d$.  Define $W_p^k:=W_p^k(\R^d)$ to be the Sobolev space of functions in $L_p$ whose first $k$ weak derivatives are in $L_p$.  
The seminorm and norm on $W_p^k(\Omega)$ may be defined as
\[|g|_{W_p^k(\Omega)}:=\underset{\multi{\gamma}= k}\max\|D^\gamma g\|_{L_p(\Omega)},\quad\text{and}\quad
 \|g\|_{W_p^k(\Omega)}:=\|g\|_{L_p(\Omega)}+|g|_{W_p^k(\Omega)},\] respectively.

Let $\schwartz$ be the space of Schwartz functions on $\R^d$, that is the collection of infinitely differentiable functions $\psi$ such that for all multi-indices $\gamma$ and $\beta$,
$$\underset{x\in\R^d}\sup\left|x^\gamma D^\beta\psi(x)\right|<\infty\;.$$

The Fourier transform of a Schwartz function $\psi$ is given by
\[
 \widehat{\psi}(\xi):=\int_{\R^d} \psi(x)e^{-i\bracket{\xi, x}}dx,\quad \xi\in\R^d,
\]
whence the inversion formula is
\[
\psi^\vee(x) = \dfrac{1}{(2\pi)^d}\dint_{\R^d}\psi(\xi)e^{i\bracket{x,\xi}}d\xi,\quad x\in\R^d.
\]
In the event that these formulas do not hold (for instance for $L_p$ functions with $p>2$), the Fourier transform should be interpreted in the sense of tempered distributions.  Let $\schwartz'$ be the
space of tempered distributions (that is, the continuous dual of $\schwartz$).  Given $T\in\schwartz'$, its Fourier transform is the tempered distribution, $\widehat{T}$, which satisfies
$\bracket{\widehat{T},\phi}=\bracket{T,\widehat{\phi}},\;\phi\in\schwartz$.  For basic facts about distributions, consult \cite{fjoshi}.  The most used fact for the subsequent analysis is that if $f\in L_p$, then it may be identified with its {\em induced distribution}, $T_f\in\schwartz'$, via $\bracket{T_f,\psi}:=\int_{\R^d} f(x)\psi(x)dx$, $\psi\in\schwartz$.  Note that the integral is well-defined due to H\"{o}lder's inequality.  Additionally, the following basic fact will be utilized implicitly throughout the sequel:
\begin{lemma}\label{LEMDistribution}
If $f,g\in L_p$ and $\widehat{T_f}=\widehat{T_g}$, then $f=g$ almost everywhere.
\end{lemma}

The proof of the lemma follows from the fact that an induced distribution is the 0 distribution precisely when the function is 0 almost everywhere.  Finally, on account of Lemma \ref{LEMDistribution}, the common abuse of notation of writing $\widehat{f}$ for $\widehat{T_f}$ will be used.

For $\sigma>0$, define $E_\sigma$ to be the class of entire functions of
exponential type $\sigma$ whose restriction to $\R^d$ has at most
polynomial growth.  Namely, $f\in E_\sigma$ if and only if there is a constant $C$ and an $N\in\N$ such that
$$|f(z)|\leq C(1+|z|)^Ne^{\sigma |Im(z)|},\quad z\in\C^d.$$
Consequently, the restriction of $f$ to $\R^d$ is a tempered distribution, and the Paley-Wiener-Schwartz Theorem (see, for example, \cite[Theorem 7.23]{Rudin}) states that the distributional Fourier transform of $f$ has
support (in the distributional sense) in the ball of radius $\sigma$ centered at the origin, which we denote $B(0,\sigma)$.  The classes $E_\sigma$ are generalizations of the traditional Paley-Wiener spaces of bandlimited functions.


Let $\alpha\in\R$ and $c>0$ be fixed; then define the \textit{general multiquadric} by
\begin{equation}\label{EQgmcdef}
\phi_{\alpha,c}(x):=\left(|x|^2+c^2\right)^\alpha,\quad x\in\R^d.
\end{equation}
The parameter $c$ is often called the {\em shape parameter} of the multiquadric.
If $\alpha\in\R\setminus\N_0$ ($\N_0$ being the natural numbers including 0), the generalized Fourier transform of $\phi_{\alpha,c}$ is given by the following (see, for example, \cite[Theorem 8.15]{Wendland}):
\[
 \phica(\xi)=\dfrac{2^{1+\alpha}}{\Gamma(-\alpha)}\left(\dfrac{c}{|\xi|}\right)^{\alpha+\frac{d}{2}}K_{\alpha+\frac{d}{2}}(c|\xi|),\quad \xi\in\R^d\setminus\{0\},
\]
where $K_\nu$ is called the modified Bessel function of the second kind (see \cite[p.376]{AandS} for its precise definition).  A few properties germane to our analysis here are that $K_\nu$ has an algebraic singularity at the origin and exponential decay away from the origin.  We note that $\phi_{\alpha,c}$, and consequently its Fourier transform, are radial functions (i.e. $\phi_{\alpha,c}(x)=\varphi_{\alpha,c}(|x|)$ for some univariate function $\varphi_{\alpha,c}$). 

Much of the results presented here pertain to {\em principal shift-invariant} spaces which are subspaces of $L_p$.  Following \cite{AG}, these can be defined as
$$V_p(\psi):=\left\{\sum_{j\in\Z^d}c_j\psi(\cdot-j):(c_j)\in\ell_p\right\},$$
where convergence of the series is taken to be in $L_p$. The function $\psi$ is often called the {\em generator}, or the window, or kernel, of the shift-invariant space.  Note also that in some of the literature, the space is defined to be the closed linear span of $\{\psi(\cdot-j):j\in\Z^d\}$ in $L_p$; however, for sufficiently nice generators, the definitions coincide.  Additionally, shift-invariant spaces may be defined for more general lattices; specifically, we will make use of $V_p(\psi,h\Z^d):=\{\sum_{j\in\Z^d}c_j\psi(\cdot-hj):(c_j)\in\ell_p\}$ in the sequel.


\section{Main Results}\label{SECMain}

For ease of viewing, the main results are all contained in this section.  We begin by setting some definitions.

A function $L:\R^d\to\R$ is a {\em cardinal function} provided $L(k)=\delta_{0,k}$, for all $k\in\Z^d$ (these are also often called Lagrange functions or fundamental functions in the literature).  There are cardinal functions associated with all manner of radial basis functions, and one common construction is to define them via their Fourier transforms.  The primary concern here is the cardinal function associated with the general multiquadric; to wit, for a fixed $\alpha\in\R\setminus\N_0$ and $c>0$, define
\[
\widehat{L_{\alpha,c}}(\xi):=\dfrac{\phica(\xi)}{\zsumd{j}\phica(\xi+2\pi j)},\quad\xi\in\R^d\setminus\{0\}.
\]
It was shown in \cite{HammLedford} that $\widehat{L_{\alpha,c}}\in L_1\cap L_2(\R^d)$; it follows that $L_{\alpha,c}:=\widehat{L_{\alpha,c}}^\vee$ is continuous, square-integrable, and a cardinal function (see Section 3 therein).

Suppose $g\in W_p^k(\R^d)$ with $k>d/p$ (thus pointwise evaluation of $g$ is well-defined since it is continuous by the Sobolev embedding theorem).  Let $h\in(0,1]$, and fix $\alpha\in(-\infty,-d-1/2)\cup[1/2,\infty)\setminus\N$.  Then formally define the multiquadric interpolant of $g$ via
\[
I_\alpha^h g(x):=\sum_{j\in\Z^d}g(hj)L_{\alpha,\frac1h}\left(\frac{x}{h}-j\right),\quad x\in\R^d.
\]
Presuming the interpolant is well-defined, it is evident that it satisfies $I_\alpha^hg(hk)=g(hk)$, $k\in\Z^d$ since $L_{\alpha,\frac1h}$ is a cardinal function.

In \cite[Theorem 8]{HammLedford}, it was shown that the univariate interpolation operator $I_\alpha^h$ is bounded from $\ell_p(\Z)\to L_p(\R)$ for every $1\leq p\leq\infty$ and $\alpha$ in the restricted range specified above.  However, the argument for higher dimensions is identical; indeed, in the course of the proof there, the authors essentially use the techniques of Jia and Micchelli \cite{JM} (see also \cite{Johnson}).  Consequently, since $g\in W_p^k(\R^d)$ implies that $(g(hj))_{j\in\Z^d}\in\ell_p$, it follows that $I_\alpha^hg\in L_p$, and particularly that $I_\alpha^h g\in V_p(L_{\alpha,\frac1h}(\cdot/h),h\Z^d)$.

It should be noted that approximation in such families of spaces have been studied extensively by Johnson and others (for example, \cite{Johnson}).  The family of subspaces $\{V_p(L_{\alpha,\frac1h}(\cdot/h),h\Z^d)\}_{h>0}$ is therein termed a {\em nonstationary ladder} of principal shift-invariant spaces.  In addition, the interested reader is referred to \cite{HoltzRon} for discussion of principal shift-invariant subspaces of $W_2^k(\R^d)$ and their approximation orders.

\subsection{Structural Results}

Before discussing the main interpolation results, we pause to mention some facts about the shift-invariant spaces associated with the multiquadric cardinal functions and how they relate to the spaces of translates of the multiquadrics themselves.  Often, RBF interpolation schemes begin by trying to find interpolants to a given class of functions from the closed linear span of translates of the RBF itself (e.g. the multiquadric), where the closure is taken, for example, in the topology of uniform convergence on compact subsets of $\R^d$.  Often in interpolation methods, the cardinal functions serve as a change of basis in the spirit of classical Lagrange polynomial interpolation, and one has $L$ as an element of the closed linear span of $\{\phi(\cdot-j):j\in\Z^d\}$.  The current analysis begins from the opposite vantage point, and considers interpolation from $V_p(L_{\alpha,c})$, and discusses how such spaces relate to their counterparts arising from $\phi_{\alpha,c}$.  

The following two theorems demonstrate that for sufficiently negative $\alpha$, we may classify the structure of the cardinal functions and the decay rate of their coefficients, as well as showing equality of the associated shift-invariant spaces $V_p(L_{\alpha,c})$ and $V_p(\phi_{\alpha,c})$.

\begin{theorem}\label{L_coeff_bnd}
Suppose that $\alpha<-d-1/2$ and $c>0$.  Then $L_{\alpha,c}$ has a series representation of the form
\[
L_{{\alpha,c}}(x)=\sum_{j\in\Z^d} a_j \phi_{\alpha,c}(x-j),
\]
where
\[
|a_j|=\begin{cases} O(|j|^{-\lfloor 2|\alpha|-d\rfloor}) & \alpha\notin\Z, \\
O(|j|^{-2|\alpha|+d+1}) & \alpha\in\Z.
\end{cases}
\]
In particular, $a\in\ell_1$ and the series is uniformly convergent.
\end{theorem}

\begin{theorem}\label{InterSpace}
Suppose that $\alpha<-d-1/2$.  Then for all $1\leq p\leq \infty$ and $c>0$,
\[
V_p(L_{\alpha,c})= V_p(\phi_{\alpha,c}).
\]
Consequently, for all $h>0$,
\[
V_p(L_{\alpha, 1/h}(\cdot/h),h\Z^d)= V_p(\phi_{\alpha,1/h}(\cdot/h),h\Z^d).
\]
\end{theorem}

Note that if $\alpha>0$, the space $V_p(\phi_{\alpha,c})$ is not well-defined since $\phi_{\alpha,c}$ is unbounded. Nonetheless, something may be said based on the analysis of Buhmann \cite{Buhmann}, which we summarize in the following remark.

\begin{rem}
For $\alpha,c>0$, $\phi_{\alpha,c}$ satisfies the admissibility conditions of Buhmann \cite{Buhmann} of order $\ceiling{2\alpha}$.  It follows then by Theorems 6, 10, and 11 therein that $L_{\alpha,c}(x)=\sum_{j\in\Z^d}a_j\phi_{\alpha,c}(x-j)$ where the series is absolutely convergent, the coefficients satisfy $|a_j|=O(|j|^{-2d-2\alpha})$, $|k|\to\infty$, and the cardinal function satisfies $|L_{\alpha,c}(x)|=O(|x|^{-2d-2\alpha}),$ $|x|\to\infty$.
\end{rem}

The proofs of the preceding theorems require some more detailed estimates of the decay of the cardinal functions, and so are postponed until Section \ref{SECProofs}.

\begin{rem}
It should be noted that the decay conditions on the cardinal functions (cf. Corollary \ref{CORmultiplierinversebounds}) imply that for all $\alpha\in(-\infty,-d-1/2)\cup[1/2,\infty)\setminus\N$, $\{L_{\alpha,\frac1h}(\cdot/h-j):j\in\Z^d\}$ is a Riesz basis for $V_2(L_{\alpha,\frac1h}(\cdot/h),h\Z^d)$, so it follows on account of \cite[Theorem 2.4]{AG} that $\{L_{\alpha,\frac1h}(\cdot/h-j):j\in\Z^d\}$ is an unconditional basis for $V_p(L_{\alpha,\frac1h}(\cdot/h),h\Z^d)$ for every $1\leq p\leq\infty$, and moreover that the latter is a closed subspace of $L_p$.  Consequently, the same is true if $L$ is replaced by $\phi$ when $\alpha<-d-1/2$ on account of Theorem \ref{InterSpace}.
\end{rem}

\subsection{Interpolation and Approximation Rates}

To begin the discussion of interpolation of $g\in W_p^k(\R^d)$ by $I_\alpha^hg\in V_p(L_{\alpha,\frac1h}(\cdot/h),h\Z^d)$, first note that existence of the interpolant is given by the boundedness of $I_\alpha^h$ as an operator from $\ell_p\to L_p$ (discussed above).  Secondly, uniqueness follows from the definition of a cardinal function (i.e. $\sum_{j\in\Z^d}c_jL_{\alpha,\frac1h}(\frac{\cdot}{h}-j)=0$ if and only if $c_j=0$ for all $j$).  From here on, for a given $g$ in the Sobolev space, $I_\alpha^h g$ is to be taken to be the unique interpolant in $V_p(L_{\alpha,\frac1h}(\cdot/h),h\Z^d)$, where it is understood that for $\alpha<-d-1/2$ and $1\leq p\leq\infty$, this must be the same as the unique interpolant from $V_p(\phi_{\alpha,\frac1h}(\cdot/h),h\Z^d)$ by Theorem \ref{InterSpace}. In the interpolation results that follow, the constants $C$ will be independent of $h$ provided that we consider a range $0<h<h_0$ for some fixed, but arbitrary $h_0$.  However, as the behavior we are primarily interested in is that for small $h$, we state the results for $0<h\leq1$ without loss of generality, and simply alert the reader here that this is not strictly necessary.

\begin{theorem}\label{THMmaintheorem}
 Suppose $\alpha\in(-\infty,-d-1/2)\cup[1/2,\infty)\setminus\N$ is fixed. 
 Let $1<p<\infty$, $k>d/p$, and $0<h\leq1$.  There exists a constant $C$, independent of $h$, so that for every $g\in
W_p^k(\R^d)$,
 $$\|I_\alpha^hg-g\|_{L_p}\leq Ch^k\|g\|_{W_p^k}.$$
 
 If $p=1$ and $k> d$, or $p=\infty$ and $k\in\N$, there is a constant $C$, independent of $h$, so
that for every $g\in W_p^k(\R^d)$,
 $$\|I_\alpha^hg-g\|_{L_p}\leq C(1+|\ln h|)h^k\|g\|_{W_p^k}.$$
\end{theorem}

The proof of the above theorem follows from an indirect argument which considers interpolation of bandlimited functions $f\in E_\frac{\pi+\eps}{h}\cap W_p^k(\R^d)$, which themselves interpolate the Sobolev functions at $h\Z^d$.  This argument follows the insightful techniques of \cite{hmnw}.  The following lemma shows that this interpolation of Sobolev functions by bandlimited ones is stable by providing Jackson and Bernstein type inequalities.

\begin{lemma}[\cite{hmnw}, Lemma 2.2]\label{LEMhmnwapproxbandlimited}
 Let $0<\varepsilon<\pi$, $1\leq p\leq\infty$, and $k>d/p$.  If $g\in W^k_p$, then given $0<h\leq1$, there is a function $f\in
E_\frac{\pi+\eps}{h}\cap W_p^k$ satisfying

 \begin{equation}\label{EQbandlimitedinterpcondition}
 f(hj) = g(hj),\quad j\in\Z^d, 
 \end{equation}

 \begin{equation}\label{EQBernstein}
  \|f-g\|_{L_p}\leq Ch^k|g|_{W_p^k},
 \end{equation}

 and
 \begin{equation}\label{EQJackson}
  |f|_{W_p^k}\leq C|g|_{W_p^k},
 \end{equation}

where $C$ is a constant independent of $h$ and $g$.
\end{lemma}

The next key ingredient to the proof of Theorem \ref{THMmaintheorem} is the following, which shows that the interpolation operators $(I_\alpha^h)_{h\in(0,1]}$ are uniformly bounded in the Sobolev seminorm.

\begin{theorem}\label{THMstabilinterpolation}
Let $\alpha\in(-\infty,-d-1/2)\cup[1/2,\infty)\setminus\N$ be fixed, and let $1<p<\infty$, $0<h\leq1$, and $k>d/p$.  There exists a constant $C$ such that for every suitably small $\eps>0$,
 \begin{equation}\label{EQseminormuniformbound}
 |I_\alpha^hf|_{W_p^k}\leq C\|f\|_{W_p^k},\qquad f\in
E_\frac{\pi+\eps}{h}\cap W_p^k(\R^d).\end{equation}
 For $p=1$ and $k> d$, or $p=\infty$ and $k\in\N$, there is a constant $C$ such that
 \begin{equation}\label{EQseminormuniformboundendpoints}|I_\alpha^hf|_{W_p^k}\leq C(1+|\ln h|)\|f\|_{W_p^k},\qquad f\in
E_\frac{\pi+\eps}{h}\cap
W_p^k(\R^d).\end{equation}
\end{theorem}

The constants in Theorems \ref{THMmaintheorem} and \ref{THMstabilinterpolation} will depend on $\alpha, p, k$, and $d$, but not on $h$.  Also note that $k$ is independent of $\alpha$ in all of the theorems stated here. The proof of Theorem \ref{THMmaintheorem} is now immediate.

\begin{proof}[Proof of Theorem \ref{THMmaintheorem}]
 Suppose $g\in W_p^k(\R^d)$ for $1<p<\infty$, and let $f\in E_\frac{\pi+\eps}{h}\cap W_p^k$ be the function provided by Lemma \ref{LEMhmnwapproxbandlimited}.  Then on account of \eqref{EQbandlimitedinterpcondition}, $I_\alpha^hg = I_\alpha^hf$, and so $\|I_\alpha^hg-g\|_{L_p}\leq \|I_\alpha^hf-f\|_{L_p}+\|f-g\|_{L_p}$.  The latter term is bounded by $Ch^k|g|_{W_p^k}$ due to \eqref{EQBernstein}.  To estimate the first term, applying a bound due to Madych and Potter \cite[Corollary 1]{MadychPotter} on the norm of $W_p^k$ functions with closely spaced zeros provides the estimate 
 $$\|I_\alpha^hf-f\|_{L_p}\leq Ch^k|I_\alpha^hf-f|_{W_p^k}\leq Ch^k\left(|I_\alpha^hf|_{W_p^k}+|f|_{W_p^k}\right),$$
 whence applying \eqref{EQJackson} and \eqref{EQseminormuniformbound} and combining the above estimates yields the desired inequality.
 
 The proof for $p=1,\infty$ is identical but for applying \eqref{EQseminormuniformboundendpoints} in the final step rather than \eqref{EQseminormuniformbound}.
\end{proof}

The proof of Theorem \ref{THMstabilinterpolation} is technical and postponed to later sections.


\section{The multiquadric multiplier}\label{SECMultiplier}


\subsection{A Note on Fourier Multipliers}

For the subsequent analysis, it is pertinent to stop for a moment and collect
some properties of {\em Fourier multiplier operators}.  Let $m$ be a
measurable function.  Then we define the linear multiplier operator $T_m$ by
$$T_m f := (m\widehat{f})^\vee,$$
which, in the event that the convolution theorem holds, is
$$T_m f = m^\vee\ast f.$$
Now {\em a priori}, it is not clear how this operator is even defined on $L_p$
for general $p$, so at the moment, consider $T_m f$ defined as above for Schwartz functions $f$.  Supposing that there is a constant such that
$$\|T_m f\|_{L_p}\leq C\|f\|_{L_p},\qquad f\in\schwartz,$$
then by density, $T_m$ extends to a bounded linear operator on $L_p$, and moreover
$$\|T_m f\|_{L_p}\leq C\|f\|_{L_p},\qquad f\in L_p.$$

Since we will be considering the same definition for the multiplier
and estimating its multiplier norm for different values of $p$, we define the
$p$--multiplier norm of $T_m$ in the natural way:
$$\|T_m\|_{\Mp}:=\|m\|_{\Mp}:=\underset{\|f\|_{L_p}=1}{\sup}\|T_m
f\|_{L_p}.$$

Notice that if $m^\vee\in L_1$, then $T_m$ is a bounded linear operator
on $L_p$.  Indeed, by Young's Inequality,
$$\|T_m f\|_{L_p} = \|m^\vee\ast
f\|_{L_p}\leq\|m^\vee\|_{L_1}\|f\|_{L_p}.$$
It follows that 
\begin{equation}\label{EQmultiplieronenormbound}
\|m\|_{\Mp}\leq\|m^\vee\|_{L_1}. 
\end{equation}
However, we can (and will) also make use of the estimation
\begin{equation}\label{EQmultiplierconvolutioninequality}
 \|m^\vee\ast f\|_{L_p}\leq\|m\|_{\Mp}\|f\|_{L_p}.
\end{equation}

\subsection{The Multiquadric Multiplier}

We now define the Fourier multiplier whose operator norm will govern much of the analysis in the sequel. Let
\begin{equation}\label{EQmultiplierdef}
 m_{\alpha,h}(\xi):=\widehat{L_{\alpha,\frac{1}{h}}}(h\xi),\quad \xi\in\R^d.
\end{equation}
Therefore, we have
\begin{equation}\label{EQmultiplierinverseFTdef}
 m_{\alpha,h}^\vee(x) = \dfrac{1}{h^d}L_{\alpha,\frac{1}{h}}\left(\frac{x}{h}\right),\quad x\in\R^d.
\end{equation}

In what follows, we consider interpolation of functions $f\in
E_\frac{\pi+\eps}{h}\cap W_p^k(\R^d)$ on account of Lemma \ref{LEMhmnwapproxbandlimited}.  First, note that \eqref{EQmultiplierinverseFTdef} implies that the interpolant of a function $f$ may be expressed as $I_\alpha^h f(x)=h^d\sum_{j\in\Z^d}f(hj)m_{\alpha,h}^\vee(x-hj)$.  Consider the following formal calculation of the Fourier transform of $I_\alpha^h$:

\begin{align}\label{EQIhFourierCalc}
\widehat{I_\alpha^hf}(\xi) & = 
h^d\left[\zsumd{j}f(hj)m_{\alpha,h}^\vee(\cdot-hj)\right]^\wedge(\xi)\nonumber\\
& = h^d\zsumd{j}f(hj)e^{-ih\bracket{j,\xi}}m_{\alpha,h}(\xi)\nonumber\\
& = \zsumd{j}\widehat{f}\left(\xi-\frac{2\pi j}{h}\right)m_{\alpha,h}(\xi).
\end{align}
In the case $p=2$, the above calculation is completely justified by the Poisson summation formula since $f$ is in the classical Paley--Wiener space. 
However, for general $p\neq2$, the above formula needs to be taken distributionally.  Indeed, denote the exponential function as $e_x:=e^{i\langle x, \cdot\rangle}$, and the translation operator on tempered distributions via $\bracket{\tau_xT,\psi}:=\bracket{T,\psi(\cdot-x)}$.  Recalling that $I_\alpha^hf\in L_p$ and thus induces a well-defined tempered distribution, the formal calculation above is as follows:

\begin{align}\label{EQIFourierDistribution}
\widehat{I_\alpha^hf} & = 
h^d\left[\zsumd{j}f(hj)m_{\alpha,h}^\vee(\cdot-hj)\right]^\wedge\nonumber\\
& = h^d\zsumd{j}f(hj)e_{-hj}m_{\alpha,h}\nonumber\\
& = \zsumd{j}\tau_{\frac{2\pi j}{h}}\widehat{f}m_{\alpha,h}.
\end{align}

Justification of this identity is somewhat more complicated.  The main idea is that the right-hand side of the second equality of \eqref{EQIFourierDistribution} is a periodic tempered distribution times an integrable (but not infinitely differentiable) function $m$.  The action of such an object on a test function is not well-defined; however, $m$ may be convolved with a standard $C^\infty$ mollifier that is an $L_1$ approximate identity, and then the right-hand side defines a tempered distribution after taking a limit as the approximate identity parameter goes to $\infty$.  Following this, the final equality stems from the Poisson summation formula for compactly supported tempered distributions \cite[Corollary 8.5.1]{fjoshi}, while the second equality is justified by the fact that the series in question converges in $L_p$, and hence $\schwartz'$.  A complete proof is given in the appendix.

In the sequel, we will use \eqref{EQIhFourierCalc} (which is an abuse of notation in the $p\neq2$ case but should cause no confusion), and consider $\mah$ as a Fourier multiplier acting on $L_p(\R^d)$.  The multiplier norm of $\mah$ for different values of $p$ will determine the behavior of the seminorm of $I_\alpha^hf$ since \eqref{EQIFourierDistribution} implies that $I_\alpha^hf = T_{m_{\alpha,h}}(\sum_{j\in\Z^d}\tau_\frac{2\pi j}{h}\widehat{f})^\vee$ in the multiplier notation of the previous subsection.

Use of Fourier multipliers to provide the seminorm estimates comes from the techniques of \cite{hmnw}.  However, unlike the Gaussian and its associated multiplier, the general multiquadrics in several variables are not tensor products of their univariate counterparts.  Consequently, to determine the properties of $m_{\alpha,h}$, it is not sufficient to consider the case $d=1$.  Nonetheless, due to the radial nature of the multiquadrics, we may still find sufficient estimates on the decay of the multiplier to prove Theorem \ref{THMstabilinterpolation}.

\subsection{Estimates for $m_{\alpha,h}$}\label{SECmestimates}

To use \eqref{EQmultiplieronenormbound}, it suffices to obtain bounds on the function $m_{\alpha,h}$.  The calculations closely resemble those used in \cite{HammLedford} to derive estimates for the cardinal functions $L_{\alpha,c}$.  For now, we restrict our attention to positive values of $\alpha$ as the calculations are essentially the same for negative values.  
Throughout the rest of this section, our calculations will be helped by the fact that \eqref{EQmultiplierdef} may be rewritten as
\[
m_{\alpha,h}(\xi)=\dfrac{\widehat{\phi_{\alpha,1}}(\xi) }{\sum_{j\in\mathbb{Z}^d} \widehat{\phi_{\alpha,1}}(\xi+\frac{2\pi j}{h})} = \dfrac{\widehat{\varphi_{\alpha,1}}(|\xi |)}{\sum_{j\in\mathbb{Z}^d} \widehat{\varphi_{\alpha,1}}(|\xi+\frac{2\pi j}{h}|)  },
\]
where $\varphi$ is the univariate function associated to $\phi$, i.e. $\phi(x)=\varphi(|x|)$.  To estimate the behavior of $m_{\alpha,h}^\vee$, it suffices to estimate the $L_1$ norms of derivatives of $m_{\alpha,h}$, which by use of Leibniz rule in the formula above, requires estimates on derivatives of $\widehat{\varphi_{\alpha,1}}(|\cdot|)$ and its reciprocal.  This study is taken up henceforth.  Since the shape parameter is always 1 in the sequel, we drop the subscript from the subsequent estimates and remind the reader that the constant $C$ below will typically depend on $\alpha,d,$ and the multi-index $\gamma$, but not on $h$.  As a matter of notation, we say that $\beta\leq\gamma$ for multi-indices $\beta$ and $\gamma$ provided $\beta_j\leq\gamma_j$ for all $j$.  Now for any multi-index $\gamma$ with $\multi{\gamma} \geq 1$,
\begin{equation}\label{chain}
D^{\gamma}\widehat{\phi_{\alpha}}( \xi )= \sum_{    \{\beta:  \beta\leq \gamma, \multi{\beta}\geq 1\}}^{\multi{\gamma}}a_\beta\widehat{\varphi_{\alpha}}^{(\multi{\beta})}(| \xi |) \Omega_{\multi{\beta}-\multi{\gamma}}(\xi),
\end{equation}
where $\Omega_l$ is a homogeneous function of degree $l$ (i.e. $\Omega_l(r\xi) = r^l\Omega_l(\xi)$ for any $r>0$).  In fact, these homogeneous functions are simply combinations of (partial) derivatives of the function $\xi\mapsto| \xi |$.  One may prove this by first noting that if $\multi{\gamma}=1$, then \eqref{chain} holds by the chain rule, and then proceeding by induction on $\multi{\gamma}$.

By \eqref{chain} and the nature of $\Omega_l$, there exists a constant $C>0$ such that
\begin{equation}\label{est1}
| D^{\gamma}\widehat{\phi_{\alpha}}( \xi ) | \leq C | \xi |^{-\multi{\gamma}} \sum_{j=1}^{\multi{\gamma}}| \xi |^{j} \widehat{\varphi_{\alpha}}^{(j)}(| \xi |). 
\end{equation}
Having found an upper bound for the derivative in terms of a univariate function, we may now recycle the estimates from Section 7 of \cite{HammLedford} by replacing $\alpha$ with $\alpha+(d-1)/2$ and taking $c=1$.  For instance, for $1\leq\multi{\gamma} < 2\alpha+d$, there is a constant $C>0$ such that
\begin{equation}\label{est2}
|D^\gamma \widehat{\phi_{\alpha}}(\xi)|\leq C  | \xi |^{-2\alpha-d-\multi{\gamma}} e^{-|\xi|}.
\end{equation}
Similarly, for $\multi{\gamma}=2\alpha+d$,
\begin{equation}\label{est3}
 |D^\gamma \widehat{\phi_{\alpha}}(\xi)|\leq  C \left\{ e^{-|\xi|}|\xi|^{-2\alpha-d}\ln(1+|\xi|^{-1}) + e^{-| \xi |} | \xi |^{-4\alpha-2d}  \right\}.
\end{equation}
The previous estimates are Equations (48) and (49) in \cite{HammLedford}. Now a calculation analogous to \eqref{chain} for the reciprocal yields (for $\multi{\gamma}\geq 1$)
\[
D^{\gamma}\left( 1/\widehat{\phi_{\alpha}}  \right)(\xi ) = \sum_{    \{\beta:  \beta\leq \gamma, \multi{\beta}\geq 1\}}^{\multi{\gamma}}a_\beta \left( 1/\widehat{\varphi_{\alpha}}  \right)^{(\multi{\beta})}(|\xi |)\Omega_{\multi{\beta}-\multi{\gamma}}(\xi),
\]
thus we get a similar estimate to \eqref{est1}:
\[
\left| D^{\gamma}\left( 1/\widehat{\phi_{\alpha}}  \right)(\xi )\right  |\leq C| \xi |^{-\multi{\gamma}}\sum_{j=1}^{\multi{\gamma}} | \xi |^{j} \left\vert\left( 1/\widehat{\varphi_{\alpha}} \right)^{(j)}(|\xi |)\right\vert.
\]
For $1\leq \multi{\gamma}< 2\alpha+d$, we obtain
\begin{equation}\label{est5}
| D^{\gamma}\left( 1/\widehat{\phi_{\alpha}}  \right)(\xi )  |\leq Ce^{|\xi|}| \xi |^{2\alpha+d -\multi{\gamma}}.
\end{equation}
 
When $\multi{\gamma}=2\alpha+d$, a logarithmic term appears in \eqref{est3} which must be handled separately.  In this case, we obtain from \cite[Eqs. (50) and (51)]{HammLedford}
\begin{equation}\label{est6}
\left|\dfrac{D^{\gamma}\widehat{\phi_{\alpha}}(\xi)}{(\widehat{\phi_{\alpha}}(\xi))^2}\right| \leq C e^{|\xi |} \left\{ | \xi|^{2\alpha+d}\ln(1+|\xi|^{-1})+1\right\}\leq C e^{|\xi|}.
\end{equation} 
The term on the left of \eqref{est6} appears after applying the Leibniz rule to the term $(1/\widehat{\varphi_{\alpha}})^{(\multi{\gamma})}(|\xi|)$.  This term is the only one containing the logarithm from \eqref{est3}, while the other terms are estimated using \eqref{est5}.  Note that these estimates rely only on the order of the derivative $\multi{\gamma}$ and $\alpha$.

Following Riemenschneider and Sivakumar \cite{RiemSiva}, for $j\neq 0$, define $a_j(\xi):= \widehat{\phi_{\alpha}}(| \xi+\frac{2\pi j}{h}  |)/ \widehat{\phi_{\alpha}}(|\xi |)$ and $s(\xi):=\sum_{j\neq 0}a_j(\xi)$; then $m_{\alpha,h}=(1+s)^{-1}$.  Our estimates once again rely on univariate estimates; however, we will also make use of the infinity norm $\|\xi \|_{\infty}:=\max\{|\xi_1|,\dots,|\xi_d|\}$. Recall that     
\begin{equation}\label{norm equiv}
 d^{-1/2}| \xi | \leq \| \xi \|_{\infty} \leq | \xi |.
 \end{equation}
To estimate $\| D^\gamma m_{\alpha,h}  \|_{L_1}$, we split $\mathbb{R}^d$ into three sections based on $j\in\mathbb{Z}^d$: 
\begin{enumerate}
\item[I.] $\| \xi \|_\infty \leq \pi/h$,
\item[II.] $\xi \in 2\pi j/h + [-\pi/h,\pi/h]^d, \quad | j |=1$, and
\item[III.] $\xi \in 2\pi j/h + [-\pi/h,\pi/h]^d \quad | j |>1$.
\end{enumerate}
These regions need further refinement to avoid the faces of the cube $[-\pi/h, \pi/h]^d$ on which we have no precise estimates on the cardinal functions except the transparent bound $|\widehat{L_{\alpha,c}}(\xi)|\leq1$, which holds for all $\xi$. The estimates for these regions are corollaries of the following lemma.
\begin{lemma}\label{LEM_Daj_ests} 
Suppose that $\alpha>0$, $1\leq \multi{\gamma}\leq 2\alpha+d$, $0<h\leq1$, and $\varepsilon\in [0,1)$.  If $\| \xi \|_\infty \leq (1-\varepsilon)\pi/h$, then
\[
|D^{\gamma}a_j(\xi)|\leq C h^{\multi{\gamma}}e^{-\varepsilon\pi/(\sqrt{d}h)}\begin{cases}1, & |j|=1  \\ e^{-2\pi|j|/(3dh)}, & |j|>1\end{cases},
\]
where $C>0$ is independent of $h$ and is uniformly bounded for $\eps\in[0,1)$.
\end{lemma} 
\begin{proof}
By using \eqref{est2},\eqref{est3},\eqref{est5}, and \eqref{est6}, we have for $1\leq \multi{\gamma}\leq 2\alpha +d$,
\begin{equation}\label{a_j}
|D^{\gamma}a_j (\xi)| \leq C e^{| \xi |-| \xi+\frac{2\pi j}{h}  |}\sum_{\{\beta: \beta\leq \gamma, \multi{\beta}\geq 1\}}\dfrac{| \xi |^{2\alpha+d-\multi{\beta}}}{\left| \xi+\frac{2\pi j}{h} \right|^{2\alpha+d-\multi{\beta}+\multi{\gamma}}}\;.
\end{equation}
 Notice that, by \eqref{norm equiv}, the summands in \eqref{a_j} satisfy
\begin{align*}
\dfrac{| \xi |^{2\alpha+d-\multi{\beta}}}{\left| \xi+\frac{2\pi j}{h} \right|^{2\alpha+d-\multi{\beta}+\multi{\gamma}}}  \leq & h^{\multi{\gamma}}\dfrac{\sqrt{d}((1-\varepsilon)\pi)^{2\alpha+d-\multi{\beta}}}{((1+\varepsilon)\pi)^{2\alpha+d-\multi{\beta}+\multi{\gamma}}} \\
\leq & Ch^{\multi{\gamma}},
\end{align*}
where $C>0$ is a constant independent of $h$.
This being the desired estimate for the sum, we turn our attention to the exponential term in \eqref{a_j}, noting that the exponent may be written as
 \begin{equation}\label{exponent}
 | \xi |-\left| \xi+\frac{2\pi j}{h}  \right| = \dfrac{-4(\pi/h)\left((\pi/h)| j |^2+\langle \xi, j\rangle \right)}{| \xi|+|\xi+\frac{2\pi j}{h} | }.
 \end{equation}
 Next, separate the integers into two cases: (i) $\| j \|_\infty =1$, and (ii) $\| j \|_\infty =M >1  $.  In case (i), the expression on the right hand side of \eqref{exponent} never changes sign because $\| \xi \|_\infty \leq (1-\varepsilon)\pi/h$.  Thus, if $j$ has $1\leq k\leq d$ non-zero components, we obtain
 \[
\dfrac{-4(\pi/h)\left((\pi/h)| j |^2+\langle \xi, j\rangle \right)}{| \xi|+|\xi+\frac{2\pi j}{h} | } \leq \dfrac{-k\varepsilon \pi}{\sqrt{d}h} \leq \dfrac{-\varepsilon \pi}{\sqrt{d}h}.
\]
Hence, 
\begin{equation}\label{j=1}
|D^\gamma a_j(\xi)|\leq C h^{\multi{\gamma}} e^{-(\varepsilon \pi)/(\sqrt{d}h)}.
\end{equation}
In case (ii), the following holds by similar reasoning to that in case (i):
\begin{align*}
\dfrac{-4(\pi/h)\left((\pi/h)| j |^2+\langle \xi, j\rangle \right)}{| \xi|+|\xi+\frac{2\pi j}{h} | } \leq &\dfrac{-2(\pi/h)\left( (M^2-M)+ \varepsilon M \right)}{\sqrt{d}(M+1-\varepsilon)} \\
 \leq &  \dfrac{-2\pi M}{3\sqrt{d}h}+ \dfrac{-\epsilon\pi}{\sqrt{d}h}.
\end{align*}
Thus we have
\begin{align*}
|D^\gamma a_j(\xi)|\leq & C h^{\multi{\gamma}}e^{-(\varepsilon \pi)/(\sqrt{d}h)}e^{-(2\pi \| j \|_\infty)/(3\sqrt{d}h)} \\
 \leq & C h^{\multi{\gamma}}e^{-(\varepsilon \pi)/(\sqrt{d}h)}e^{-(2\pi | j |)/(3dh)} .
\end{align*}
\end{proof}

\begin{corollary}\label{Dm_est1}
Suppose that $\alpha>0$, $1\leq \multi{\gamma}\leq 2\alpha+d$, $0<h\leq1$, and $\varepsilon\in [0,1)$.  If $\| \xi \|_\infty \leq (1-\varepsilon)\pi/h$, then
\[
|D^{\gamma}m_{\alpha,h}(\xi)|\leq C h^{\multi{\gamma}}e^{-\varepsilon\pi/(\sqrt{d}h)},
\]
where $C>0$ is independent of $h$.
\end{corollary}

\begin{proof}
From Lemma \ref{LEM_Daj_ests}, we see that
\begin{align*}
 |D^\gamma s(\xi)|\leq & C h^{\multi{\gamma}} e^{-(\varepsilon \pi)/(\sqrt{d}h)} \left[ (3^d-1)+ \sum_{|j|\geq 2} e^{-(2\pi | j|)/(3dh)}  \right]\\
 \leq & C h^{\multi{\gamma}}e^{-(\varepsilon\pi)/(\sqrt{d}h)};
\end{align*}
by noting that $s(\xi)\geq 0$ and applying the Leibniz rule, we find that
\[
|D^\gamma m_{\alpha,h}(\xi)|\leq C h^{\multi{\gamma}}e^{-(\varepsilon\pi)/(\sqrt{d}h)}.
\]
\end{proof}

\begin{corollary}\label{Dm_est2}
Suppose that $\alpha>0$, $1\leq \multi{\gamma}\leq 2\alpha+d$, and $0<h\leq1$.  If $ \xi\in 2\pi j/h+[-\pi/h,\pi/h]^d$, where $|j|>1$, then
\[
|D^{\gamma}m_{\alpha,h}(\xi)|\leq C h^{\multi{\gamma}}e^{-2\pi|j|/(3dh)},
\]
where $C>0$ is independent of $h$.
\end{corollary}
\begin{proof}
 We can write $\xi=2\pi j/h + r$, where $r\in [-\pi/h,\pi/h]^d$, and we have 
\[
m_{\alpha,h}(\xi) = a_{j}(r)m_{\alpha,h}(r).
\]
Consequently, we can use the estimates in region I (Lemma \ref{LEM_Daj_ests}) with $\varepsilon=0$ and the Leibniz rule to obtain
\[
|D^\gamma m_{\alpha,h}(\xi)|\leq C h^{\multi{\gamma}} e^{-(2\pi |j|)/(3dh)}.
\]
\end{proof}
As mentioned before, on a face of the cube $[-\pi/h,\pi/h]^d$, we must be more careful due to the lack of bounds on the Fourier transform of the cardinal function there.  To avoid these we introduce the following regions when $|j|=1$:
\begin{align*}
    R_{i+}&:=\left\{(\xi_1,\dots,\xi_d):-(1-\varepsilon)\pi/h\leq \xi_i \leq 3\pi/h, |\xi_k|\leq\pi/h, k\neq i    \right\} \text { and }\\
    R_{i-}&:=\left\{ (\xi_1,\dots,\xi_d):-3\pi/h\leq\xi_i\leq -(1-\varepsilon)\pi/h, |\xi_k|\leq\pi/h, k\neq i    \right\}.
\end{align*}
Now we define $R_j$ to be $R_{\pm i}$ where the plus or minus is chosen to match the $i^{th}$ component of $j$.
\begin{corollary}\label{Dm_est3}
Suppose that $\alpha>0$, $1\leq \multi{\gamma}\leq 2\alpha+d$, $0<h\leq1$, and $\varepsilon\in [0,1)$.  If $|j|=1$ and $\xi \in R_{j}$, then
\[
|D^{\gamma}m_{\alpha,h}(\xi)|\leq C h^{\multi{\gamma}}e^{-\varepsilon\pi/(\sqrt{d}h)},
\]
where $C>0$ is independent of $h$.
\end{corollary}
\begin{proof}
This estimate follows similarly to that for Corollary \ref{Dm_est2}, except we must be careful when a component of $j$ approaches $\pm \pi$.  We will outline what to do in the case that $j=e_1=(1,0,\dots,0)$ and omit the other cases since they are handled in a completely similar manner.  Consider $\xi\in [(1+\varepsilon)\pi,3\pi]\times[-\pi,\pi]^{d-1}$.  We may write $\xi=2\pi e_1/h +r $, where $r\in [-(1-\varepsilon)\pi,\pi]\times[-\pi,\pi]^{d-1}$, and again we have
\[
m_{\alpha,h}(\xi)=a_{e_1}(r)m_{\alpha,h}(r).
\]
For this region, the estimate in \eqref{j=1} still holds, so we may apply the Leibniz rule to obtain
\[
|D^\gamma m_{\alpha,h}(\xi)|\leq C h^{\multi{\gamma}}e^{-(\varepsilon\pi)/(\sqrt{d}h)}.
\]

\end{proof}
Combining these estimates yields the main result in this section.
\begin{theorem}\label{DmL1}
Suppose that $\alpha>0$, $0<h\leq 1$ and $1\leq \multi{\gamma}\leq 2\alpha+d$.  Then there exists a constant $C>0$, independent of $h$, such that
\[
\| D^\gamma m_{\alpha,h}  \|_{L_1(\mathbb{R}^d)}\leq C h^{\multi{\gamma}}.
\]
\end{theorem}
\begin{proof}
For $1\leq \multi{\gamma} \leq 2\alpha+d$, by combining estimates from Corollaries \ref{Dm_est1}, \ref{Dm_est2}, \ref{Dm_est3}, and letting $m:=\min\left\{\dfrac{\varepsilon\pi}{\sqrt{d}h},\dfrac{2\pi}{3dh} \right\} $, we have
\[
\| D^\gamma m_{\alpha,h}  \|_{L_1(\mathbb{R}^d)}\leq C\sum_{j\in\Z^d} h^{\multi{\gamma}-d}e^{-m|j|}.
\]
Estimating this sum via the integral definition of the Gamma function yields the result.
\end{proof}

The corresponding result for sufficiently negative $\alpha$ is the following (the analogues of the univariate estimates from \cite{HammLedford} follow by replacing $\alpha$ with $|\alpha|-1$, as noted in Eq. (44) therein).
\begin{theorem}\label{-DmL1}
Suppose that $\alpha<-d-1/2$, $0<h\leq 1$ and $1\leq \multi{\gamma}<2|\alpha|-d$, then there exists a constant $C>0$, independent of $h$, such that
\[
\| D^\gamma m_{\alpha,h}  \|_{L_1(\mathbb{R}^d)}\leq Ch^{\multi{\gamma}}.
\]
\end{theorem}

Theorems \ref{DmL1} and \ref{-DmL1} together with \eqref{EQmultiplierinverseFTdef} and the fact that $|L_{\alpha,c}(x)|\leq1$ provide the following estimates for $m_{\alpha,h}^\vee$.

\begin{corollary}\label{CORmultiplierinversebounds}
Suppose $\alpha\in(-\infty,-d-1/2)\cup[1/2,\infty)\setminus\N$ and $0<h\leq1$.  There exists a constant $C>0$, independent of $h$, such that
$$|m_{\alpha,h}^\vee(x)|\leq C\min\left\{\frac{1}{h^d},\frac{1}{|x|^d},\frac{1}{|x|^{d+1}}\right\}.$$
\end{corollary}


\subsection{The Multiplier Norm of $m_{\alpha,h}$}

We now consider the behavior of $\|m_{\alpha,h}\|_{\Mp}$ for different values of $p$. 
This will aid the analysis in proving the stability of the interpolation operator $I^h_\alpha$.  

\begin{theorem}\label{THMmultiplierbounded}
 Suppose that $\alpha\in (-\infty,-d-1/2)\cup [1/2,\infty)\setminus \mathbb{N} $ and $0<h\leq1$.  Then for each $1<p<\infty$, there exists a constant $C>0$ such that
 $$\|m_{\alpha,h}\|_{\mathcal{M}_p}\leq C.$$
 If $p=1,\infty$, then there exists a constant $C>0$ such that $$\|m_{\alpha,h}\|_{\Mp}\leq C\left(1+|\ln h|\right).$$
\end{theorem}

\begin{proof}
  First, notice that $\|m_{\alpha,h}\|_{\mathcal{M}_1} = \|m_{\alpha,h}\|_{\mathcal{M}_\infty}\leq \|m_{\alpha,h}^\vee\|_{L_1(\R^d)}$.
    By Corollary \ref{CORmultiplierinversebounds}, the stated upper bound for $\| m_{\alpha,h}^\vee  \|_{L_1(\mathbb{R}^d)}$ is obtained as follows:
\begin{align*}
  &C\left[\int_{0}^{h} h^{-d}r^{d-1}dr + \int_{h}^{1} r^{-1}dr +   \int_{1}^{\infty}  r^{-2}dr   \right]\\
   \leq & C\left( 1  + |\ln(h)|       \right).
\end{align*}
 
 If $1<p<\infty$, the conclusion of the theorem follows directly from the Mikhlin multiplier theorem \cite[Theorem 2, p. 232]{Mikhlin}, which states that if $\underset{x\in\R^d}\sup|x|^{[\gamma]}|D^\gamma m_{\alpha,h}(x)|\leq C$ for every $[\gamma]\leq d$, then $\|m_{\alpha,h}\|_{\mathcal{M}_p}\leq C$.  This bound follows directly from the estimates in the previous subsections.
\end{proof}


\section{Proofs of Main Theorems}\label{SECProofs}

In this section, we enumerate the proofs of the main theorems in Section \ref{SECMain}, excepting Theorem \ref{THMmaintheorem} which was already proven there.

\subsection{Proofs of Structural Theorems}

\begin{proof}[Proof of Theorem \ref{L_coeff_bnd}]
Let $c>0$ be fixed, and we drop the subscript for ease of notation.  Begin by defining the periodic symbol
\[
P_\alpha(\xi):=(\widehat{\phi_\alpha}(\xi))^{-1}\widehat{L_{\alpha}}(\xi)=\left(\sum_{j\in\Z^d} \widehat{\phi_\alpha}(\xi+2\pi j)  \right)^{-1}.
\]
From \eqref{est5}, \eqref{est6}, and Theorem \ref{-DmL1}, we see that $P_\alpha$ has derivatives of up to order $k$ in $L_1(\T^d)$, where $k<2|\alpha|-d$, and $\T^d$ is the $d$--dimensional torus, which may be identified with $[-\pi,\pi)^d$.  Since $P_\alpha\in L_1(\T^d)$, it may be identified with its Fourier series, and we have $\widehat{L_\alpha}(\xi)=\widehat{\phi_\alpha}(\xi)\sum_{j\in\Z^d}a_je^{-i\bracket{j,\xi}},$ where
\[
a_j=\frac{1}{(2\pi)^d}\int_{\T^d}P_\alpha(\xi)e^{i\langle \xi,j\rangle}d\xi.
\]
The series representation of $L_\alpha$ is immediate, and we can use integration by parts and the periodicity of $P_\alpha$ to see the decay rates of the coefficients $a_j$.  These estimates together with the decay of $\phi_\alpha$ allow us to conclude that the series is uniformly convergent. 
\end{proof}

The following lemma will be useful for the proof of Theorem \ref{InterSpace}.  Before stating it, let us note that $A(\T^d)$ is the Wiener algebra of $L_1(\T^d)$ functions with absolutely summable Fourier coefficients.

\begin{lemma}\label{LEMAT}
Let $c>0$ and $\alpha<-d-1/2$ be fixed.  Let $P_\alpha$ be the function defined in the proof of Theorem \ref{L_coeff_bnd}.  Then $P_\alpha$ and $1/P_\alpha$ are in $A(\T^d)$.
\end{lemma}

\begin{proof}
Note that Theorem \ref{L_coeff_bnd} and its proof imply that $P_\alpha\in A(\T^d)$. Consequently, if $P_\alpha(\xi)$ is bounded away from 0 on $\T^d$, then $P_\alpha$ satisfies the conditions of Wiener's $1/f$ Theorem, and $1/P_\alpha\in A(\T^d)$.  Thus, it suffices to demonstrate that for fixed $\alpha<-d-1/2$ and $c>0$, $\sum_{j\in\Z^d}|\widehat{\phi_{\alpha,c}}(\xi+2\pi j)|\leq C$, uniformly for $\xi\in\T^d$.

First, note that for this range of $\alpha$, $\phi_{\alpha,c}\in L_1\cap L_2(\R^d)$, which implies that $\vert\widehat{\phi_{\alpha,c}}(\xi)\vert\leq C$ for all $\xi$.  Thus, it suffices to bound $\sum_{j\neq0}|\widehat{\phi_{\alpha,c}}(\xi+2\pi j)|$.  Note that the quantity in question is majorized by $|\widehat{\phi_{\alpha,c}}(\xi)|\sum_{j\neq0}e^{-c(|\xi+2\pi j|-|\xi|)}$ on account of the estimates in \cite[Section 5.1]{Wendland} (see also \cite[Lemma 1]{HammLedford}; however, the rational term appearing there should be replaced by 1 due to a typographical error in the definition of $K_\nu$ therein).  Finally, this series is uniformly bounded for $\xi\in\T^d$ (see, for example, the proof of \cite[Proposition 2.2]{Baxter}), whence the proof is complete.
\end{proof}

\begin{proof}[Proof of Theorem \ref{InterSpace}]
To begin, we show that $V_p(L_{\alpha,c})\subset V_p(\phi_{\alpha,c}).$  On account of Theorem \ref{L_coeff_bnd}, we may write $f\in V_p(L_{\alpha,c}) $ as
\[
f(x)=\sum_{j\in\Z^d}b_j \sum_{k\in\Z^d}a_k\phi_{\alpha,c}(x-j-k),
\]
where $(a_k)\in\ell_1$ are the Fourier coefficients of $P_\alpha$, and $(b_j)\in\ell_p$.

The following calculation justifies the use of Fubini's Theorem, allowing us to switch the order of summation:
\begin{align*}
    &\sum_{j,k\in\Z^d}|b_j a_k||\phi_{\alpha,c}(x-j-k)|\\ 
    &=\sum_{k\in\Z^d}|a_k|\sum_{j\in\Z^d}|b_j||\phi_{\alpha,c}(x-j-k)|\\
    &\leq \| a\|_{\ell_1}\|b \|_{\ell_p} \| (\phi_{\alpha,c}(x-\cdot) ) \|_{\ell_q},
\end{align*}
where we have used H\"{o}lder's inequality.  The last term is bounded uniformly in $x$ on account of the decay of $\phi_{\alpha,c}$ and the fact that $ a \in\ell_1$.  Then we have, by re-indexing,
\begin{align*}
f(x)&=\sum_{k\in\Z^d} \sum_{j\in\Z^d}a_kb_j\phi_{\alpha,c}(x-j-k)\\
&=\sum_{m\in\Z^d}\left(\sum_{j\in\Z^d}b_j a_{m-j}\right)\phi_{\alpha,c}(x-m)\\
&=\sum_{m\in\Z^d}(a*b)_m \phi_{\alpha,c}(x-m).
\end{align*}
Now the discrete version of Young's convolution inequality implies that $b*a\in\ell_p$, hence $f\in V_p(\phi_{\alpha,c})$. The statement about $V_p(L_{\alpha,\frac1h}(\cdot/h),h\Z^d)$ follows from a simple dilation argument.

To see the reverse inclusion $V_p(\phi_{\alpha,c})\subset V_p(L_{\alpha,c})$, let $f(x) = \sum_{j\in\Z^d}b_j\phi_{\alpha,c}(x-j)$, and let $(d_k)\in\ell_1$ be the Fourier coefficients of $1/P_\alpha$ on account of Lemma \ref{LEMAT}.  The same calculation as above verifies that $$f(x) = \sum_{m\in\Z^d}(d\ast b)_mL_{\alpha,c}(x-m),$$ where the only change needed is that, in the justification of the use of Fubini's Theorem, we need $\|(L_{\alpha,c}(x-\cdot))\|_{\ell_q}$ to be finite.  Note that it suffices to bound the $\ell_1$ norm of the sequence in question, whereby on account of Corollary \ref{CORmultiplierinversebounds} and \eqref{EQmultiplierinverseFTdef}, we have
$$\|(L_{\alpha,c}(x-\cdot))\|_{\ell_1} = O\left(\sum_{j\in\Z^d}\frac{1}{1+|x-j|^{d+1}}\right)=O(1),$$ where the implicit constant is independent of $x$.
Note also that the univariate proof of this bound is \cite[Proposition 6]{HammLedford}.

\end{proof}

\subsection{Proof of Theorem \ref{THMstabilinterpolation} - Univariate Case}\label{SECstabilityunivariate}

We may now prove Theorem \ref{THMstabilinterpolation}.  Since the essence of the argument is contained in the univariate proof, we carefully consider the behavior of the interpolation operator $I^h_\alpha$ acting on functions in $E_\frac{\pi+\eps}{h}\cap W_p^k(\R).$ The changes necessary to complete the multivariate version follow in Section \ref{SECstabilitymultivariate}.    


We begin with a formula that follows similarly to \eqref{EQIhFourierCalc}:
\begin{equation}\label{EQIhDerivativeFourierCalc}
\widehat{\left(I_\alpha^hf\right)^{(k)}}(\xi) =
(i\xi)^k\left[\zsum{j}\widehat{f}\left(\xi-\frac{2\pi
j}{h}\right)\right]m_{\alpha,h}(\xi) =:\zsum{j}\widehat{G_{k,j}}(\xi).\end{equation}

We define $G_{k,j}:=D^k[(e^{2\pi i(\cdot)j/h}f)\ast m_{\alpha,h}^\vee],$ and notice that
the convolution theorem holds as $f$ is bandlimited (see, for example, \cite[Theorem 8.4.2]{fjoshi}), and so
$\widehat{G_{k,j}}$ as defined by \eqref{EQIhDerivativeFourierCalc} (cf. \eqref{EQTemp}) is indeed
the distributional Fourier transform of $G_{k,j}$.  We will prove that
$\zsum{j}\|G_{k,j}\|_{L_p}<\infty$.  This implies that 
$\left(I_\alpha^hf\right)^{(k)} = \zsum{j}G_{k,j}$ in $\schwartz'$, and 
$|I_\alpha^hf|_{W_p^k}\leq\zsum{j}\|G_{k,j}\|_{L_p}.$  We break these estimates into
three parts: $j=0$, $j=\pm1$, and $|j|>1$.

{\em Case $j=0$}:  By \eqref{EQmultiplierconvolutioninequality}, $\|D^kf\ast
m_{\alpha,h}^\vee\|_{L_p}\leq|f|_{W_p^k}\|m_{\alpha,h}\|_{\Mp}.$ So by Theorem
\ref{THMmultiplierbounded},
$$\|G_{k,0}\|_{L_p}\leq\left\{
\begin{array}{ll}
 C|f|_{W_p^k} & 1<p<\infty\\
 C(1+|\ln h|)|f|_{W_p^k} & p=1,\infty.
\end{array}\right.$$

{\em Case $|j|>1$}:

Let $\nu$ be a non-negative $C^\infty$ bump function such that $\nu(\xi)=0$
whenever $|\xi|>2\varepsilon$, $\nu(\xi)=1$ whenever $|\xi|<\varepsilon$, and
$\nu(-\xi)=\nu(\xi)$ for all $\xi$.  We use $\nu$ to construct another bump function, $\psi$, with support in
$[-\pi-2\varepsilon,\pi+2\varepsilon]$ satisfying $\psi(t) = 1$ for $-\pi\leq
t\leq\pi$, and $\psi(-t)=\psi(t)=\nu(t-\pi)$ for $t\geq\pi$.

Now since $f$ is bandlimited with band $\frac{\pi+\varepsilon}{h}$, we may write

$$\widehat{G_{k,j}}(\xi) =  \widehat{f}\left(\xi-\frac{2\pi
j}{h}\right)(i\xi)^k\psi\left(h\left(\xi-\frac{2\pi
j}{h}\right)\right)m_{\alpha,h}(\xi),$$
because on the support of $\widehat{f}$, $\psi\left(h\left(\xi-\frac{2\pi
j}{h}\right)\right) = 1$.
Now define the function
$$\rho(\xi):=\rho_{k,j}(\xi):=(i\xi)^k\psi\left(h\left(\xi-\frac{2\pi
j}{h}\right)\right)m_{\alpha,h}(\xi).$$
Using \eqref{EQmultiplieronenormbound} and
\eqref{EQmultiplierconvolutioninequality}, we may bound $\|G_{k,j}\|_{L_p}$
above by $\|\rho^\vee\|_{L_1}\left\|\left(\widehat{f}\left(\cdot-\frac{2\pi
j}{h}\right)\right)^\vee\right\|_{L_p}$ by viewing $\rho_{k,j}$ as an $L_p$ multiplier operator.  We now estimate $\|\rho^\vee\|_{L_1}$.

\begin{lemma}\label{LEMtauestimate}
Let $k\in\N$ and $\alpha\in[1/2,\infty)$.  There is a constant $C$ depending on $\eps$ and $k$ such that the following hold: 

\begin{itemize}
\item[(i)] If $|j|\geq3$, then $$\|\rho_{k,j}^\vee\|_{L_1}\leq C\left|\frac{\pi(2j+1)+2\eps}{h}\right|^{k+1}h^2e^{-\frac{2\pi}{3h}(|j|-2)}.$$
\item[(ii)] If $|j|=2$, then
$$\|\rho_{k,j}^\vee\|_{L_1}\leq C\left|\frac{\pi(2j+1)+2\eps}{h}\right|^{k+1}h^2e^{-\frac{\pi}{2h}}.$$
\end{itemize}
\end{lemma}

\begin{proof}
As $\rho$ is continuous and square-integrable, we may make use of the fact that if $$\max\left\{\dint_\R |\rho''(\xi)|d\xi,\dint_\R|\rho(\xi)|d\xi\right\}\leq M,$$
then
$|\rho^\vee(x)|\leq\dfrac{M}{\left(1+|x|^2\right)}$, whence $\|\rho^\vee\|_{L_1}\leq \pi M$.

Notice first that due to the support of $\psi$, $\rho_{k,j}(\xi)=0$ outside the interval $$I_j:=\left[\frac{(2j-1)\pi-2\eps}{h},\frac{(2j+1)\pi+2\eps}{h}\right].$$ Consequently, 

$$\dint_\R\left|\rho''_{k,j}(\xi)\right|d\xi
\leq 2\left(\frac{\pi+2\eps}{h}\right)\left\|\dfrac{d^2}{d\xi^2}\left[\xi^k\psi\left(h\xi-2\pi j\right)m_{\alpha,h}(\xi)\right]\right\|_{L_\infty(I_j)}.$$
The second derivative term in the $L_\infty$ norm above is equal to
$$\sum_{\multi{\gamma}=2}C_\gamma D^{\gamma_1}(\xi^k)D^{\gamma_2}\left(\psi\left(h\xi-2\pi j\right)\right)D^{\gamma_3}(m_{\alpha,h}(\xi)),$$
where $\gamma=(\gamma_1,\gamma_2,\gamma_3)$ is a multi-index, and $C_\gamma = \frac{2!}{\gamma_1!\gamma_2!\gamma_3!}$.

Notice that $|D^{\gamma_1}(\xi^k)|\leq C|\xi|^{k-\gamma_1}\leq C\left|\frac{\pi+2\pi j+ 2\eps}{h}\right|^k$, and also that
$$\left\|D^{\gamma_2}\left(\psi\left(h\xi-2\pi j\right)\right)\right\|_{L_\infty(I_j)}\leq Ch^{\gamma_2},$$
since derivatives of $\psi$ are again $C^\infty$ functions.

Consequently, the dominating terms will be those of the form $\xi^k\psi(h\xi-2\pi j)m_{\alpha,h}''(\xi).$  It remains then to consider how large $\|m_{\alpha,h}''(\xi)\|_{L_\infty(I_j)}$ can be.

From Corollary \ref{Dm_est2}, we find that if $|j|\geq 3$, then 
\[
\|m_{\alpha,h}''\|_{L_\infty(I_j)}\leq C h^2 e^{-\frac{2\pi(|j|-1)}{3h}}.
\]
Similarly, if $|j|=2$, then Corollary \ref{Dm_est3} with $\eps=1/2$ implies
\[
\|m_{\alpha,h}''\|_{L_\infty(I_j)}\leq C h^2e^{-\frac{\pi}{2h}}.
\]
The conclusion of the lemma follows from combining the above estimates.

\end{proof}

Now for the terms corresponding to $j=\pm1$, we must be more careful, because the estimates in Section \ref{SECmestimates} do not give uniform bounds (in terms of $h$) near the boundary points $\pm\pi$.  
Nonetheless, we may make a slight modification to the above argument to achieve our purposes.

If $j=1$, then we decompose the bump function $\psi$ as the sum of two bump functions.  If $\nu$ is the original bump function considered in the case of $|j|>1$, then let $\psi(\xi) = \omega(\xi)+\nu(\xi+\pi)$, where $\omega$ is a bump function whose support lies in $[-\pi+\eps,\pi+2\eps]$, and $\omega(\xi)=1$ whenever $\xi\in[-\pi+2\eps,\pi+\eps]$.  Note that this requires $\omega(\xi)+\nu(\xi)=1$ on the small overlapping interval $[-\pi+\eps,-\pi+2\eps]$, but this is no problem. 

If $j=-1$, then make a simple modification and write $\psi(\xi)=\overset{\sim}\omega(\xi)+\nu(\xi-\pi)$, where $\overset{\sim}\omega$ is defined with asymmetric support in $[-\pi-2\eps,\pi+\eps]$.  Since the proof is the same, we deal only with the case $j=1$.  

Similar to the previous case, we may write

\begin{equation}\label{EQGk1}
\widehat{G_{k,1}}(\xi) = \widehat{f}\left(\xi-\frac{2\pi }{h}\right)\sigma_{k}(\xi) + \widehat{f}\left(\xi-\frac{2\pi }{h}\right)\widetilde\sigma_{k}(\xi),
\end{equation}

where

$$\sigma_{k}(\xi) := (i\xi)^k\omega\left(h\xi-2\pi \right)m_{\alpha,h}(\xi),$$
and
$$\widetilde\sigma_{k}(\xi) := (i\xi)^k\nu\left(h\xi-\pi \right)m_{\alpha,h}(\xi).$$

Let us first analyze $\widetilde\sigma_{k}$ by identifying $\xi^k$ with its Taylor series.  We obtain

\begin{align}\label{EQsigma}
\widetilde\sigma_{k}(\xi) &= i^k\finsum{l}{0}{k}\binom{k}{l}\left(\frac{2\pi}{h}\right)^{k-l}\left(\xi-\frac{2\pi}{h}\right)^{l}\nu\left(h\xi-\pi\right)m_{\alpha,h}(\xi)\nonumber\\
& = i^k\left(\xi-\frac{2\pi}{h}\right)^km_{\alpha,h}(\xi)\finsum{l}{0}{k}\binom{k}{l}\left(\frac{2\pi}{h}\right)^{k-l}\frac{\nu(h\xi-\pi)}{(\xi-\frac{2\pi}{h})^{k-l}}\nonumber\\
& =: i^k\left(\xi-\frac{2\pi}{h}\right)^km_{\alpha,h}(\xi)\mu_k(\xi).
\end{align}

Consider $\mu_k$ as a Fourier multiplier, and note that
\begin{equation}\label{EQMu}\|\mu_k\|_{\mathcal{M}_p}\leq\finsum{l}{0}{k}\binom{k}{l}(2\pi)^{k-l}\left\|\left(\dfrac{\nu(\cdot-\pi)}{(\cdot-2\pi)^{k-l}}\right)^\vee\right\|_{L_1}\leq C,\end{equation}
where $C$ is a constant depending only on $k$.  This follows because the function $\nu(\cdot-\pi)/(\cdot-2\pi)^{k-l}$ belongs to $\schwartz$ since the support of $\nu(\cdot-\pi)$ is bounded away from the point $2\pi$.

Of course, the case $j=-1$ is essentially the same.  We obtain the same upper bound on the modified version of $\mu_k$.  We now turn to the multiplier $\sigma_{k}$ for $j=\pm1$.  

\begin{lemma}\label{LEMgammaestimate}
Let $k\in\N$ and $\alpha\in[1/2,\infty)$.  If $j=\pm1$, then there is a constant $C$ depending on $\eps$ and $k$ such that the following holds.

$$\|\sigma_{k}^\vee\|_{L_1}\leq C\left|\frac{\pi(2j+1)+2\eps}{h}\right|^{k+1}h^2e^{-\frac{\pi\eps}{h}}.$$

Consequently, there is a constant $C$, independent of $h$, such that
$$\|\sigma_{k}^\vee\|_{L_1}\leq C.$$
\end{lemma}

\begin{proof}
Mimic the proof of Lemma \ref{LEMtauestimate}, except note that in this case, $\supp(\omega)\subset[-\pi+\eps,\pi+2\eps]$, and thus we need to estimate $\|m''_{\alpha,h}\|_{L_\infty(\tilde{I}_j)}$, where $\tilde{I}_j:=[\frac{\pi+\eps}{h},\frac{3\pi+2\eps}{h}].$  This term is majorized by $Ch^2e^{-\frac{\eps\pi}{h}}$ on account of Corollary \ref{Dm_est3}, which provides the stated bound.
\end{proof}


We are now ready to supply the remainder of the proof.

\begin{proof}[Proof of Theorem \ref{THMstabilinterpolation}]

First let us consider the case when $1<p<\infty$ and $\alpha\in[1/2,\infty)$.  By the discussion at the beginning of this section, we find that

\begin{align}\label{EQIhseminormbound}
|I_\alpha^hf|_{W_p^k} &\leq\zsum{j}\|G_{k,j}\|_{L_p}\\\
& = \|G_{k,0}\|_{L_p}+\underset{|j|=1}\dsum\|G_{k,j}\|_{L_p}+\underset{|j|\geq2}\dsum\|G_{k,j}\|_{L_p}\nonumber\\
& =: \Sigma_1+\Sigma_2+\Sigma_3.\nonumber
\end{align}

We already saw that $\Sigma_1\leq C|f|_{W_p^k}$.  It follows from Theorem \ref{THMmultiplierbounded}, Lemma \ref{LEMgammaestimate}, \eqref{EQGk1}, \eqref{EQsigma}, and \eqref{EQMu}, that
\begin{align}\label{EQSigma2bound}
\Sigma_2 &\leq2\left\|\left(i^k\left(\cdot-\frac{2\pi}{h}\right)^k\widehat{f}\left(\cdot-\frac{2\pi}{h}\right)\right)^\vee\right\|_{L_p}\|m_{\alpha,h}\|_{\mathcal{M}_p}\|\mu_k\|_{\mathcal{M}_p}+\|f\|_{L_p}\|\sigma_{k}^\vee\|_{L_1}\nonumber\\
& \leq C\|f\|_{W_p^k}.
\end{align}

Now for the final term, Lemma \ref{LEMtauestimate} implies the following:

\begin{align*}
\Sigma_3 &\leq \|f\|_{L_p}\left[\underset{|j|=2}\dsum\|\rho_{k,j}^\vee\|_{L_1}+\underset{|j|\geq3}\dsum\|\rho_{k,j}^\vee\|_{L_1}\right]\nonumber\\
& \leq C\|f\|_{L_p}+C\|f\|_{L_p}\underset{|j|\geq3}\dsum\left|\frac{\pi(2j+1)+2\eps}{h}\right|^{k+1}h^2e^{-\frac{2\pi}{3h}(|j|-1)}.
\end{align*}

The sum over $j$ on the right hand side is majorized by a constant (depending on $\eps$ and $k$) times $h^{-k+1}e^{-\frac{2\pi}{3h}}$, which in turn is bounded by a constant depending only on $k$.  Consequently, 
\begin{equation}\label{EQSigma3bound}
\Sigma_3\leq C\|f\|_{L_p}.
\end{equation}

Combining \eqref{EQIhseminormbound}, \eqref{EQSigma2bound}, and \eqref{EQSigma3bound} provides the theorem.  Whenever $\alpha\in(-\infty,-d-1/2)$, the only change is in the use of the analogues of Lemmas \ref{LEMtauestimate} and \ref{LEMgammaestimate}.  But the bound will still be some power of $h$ times a decaying exponential, and so the corresponding series from \eqref{EQIhseminormbound} will satisfy the same upper bound up to a different constant $C$.

The proof when $p=1,\infty$ for either range of $\alpha$ is the same except that $\Sigma_1,\Sigma_2\leq C(1+|\ln h|)\|f\|_{W_p^k}$ due to the estimate on $\|m_{\alpha,h}\|_{\mathcal{M}_1}=\|m_{\alpha,h}\|_{\mathcal{M}_\infty}$ (Theorem \ref{THMmultiplierbounded}).
\end{proof}


\subsection{Proof of Theorem \ref{THMstabilinterpolation} - Multivariate Case}\label{SECstabilitymultivariate}

Similar to the univariate case, if $f\in E_\frac{\pi+\eps}{h}\cap W_p^k(\R^d)$, then we have, for $\multi{\gamma}\leq k$,
\[
\widehat{D^\gamma I_\alpha^h f}(\xi) = i^{\multi{\gamma}}\xi^\gamma\left[\zsumd{j}\widehat{f}\left(\xi-\frac{2\pi j}{h}\right)\right]\mah(\xi) =:\zsumd{j}\widehat{G_{\gamma,j}}(\xi),
\]
with $G_{\gamma,j} = D^\gamma[(fe^{i\bracket{\frac{2\pi j}{h},\cdot}}\ast\mah^\vee].$  Consequently,
$|I_\alpha^hf|_{W_p^k}\leq\underset{\multi{\gamma}= k}\sum\;\underset{j\in\Z^d}\sum\|G_{\gamma,j}\|_{L_p}$.  Thus we estimate the norms of $G_{\gamma,j}$ for different values of the given parameters.

{\em Case} $j=0$: By \eqref{EQmultiplierconvolutioninequality} and Theorem \ref{THMmultiplierbounded}, we have
\begin{displaymath}
\|G_{\gamma,0}\|_{L_p}\leq \|\mah\|_{\Mp}|f|_{W_p^k}\leq \left\{
\begin{array}{lll}
C|f|_{W_p^k} & 1<p<\infty\\
C(1+|\ln h|)|f|_{W_p^k} & p=1,\infty.\\
\end{array}\right.
\end{displaymath}

{\em Case} $|j|>1$: Let $\nu$ and $\psi$ be the smooth bump functions from Section \ref{SECstabilityunivariate}, and let $\widetilde{\psi}(\xi) = \psi(|\xi|)$, which will be a smooth bump function with support in the ball $B(0,\pi+2\eps)$ taking value 1 on $B(0,\pi+\eps)$.  Since $\supp(\widehat{f})\subset B\left(0,\frac{\pi+\eps}{h}\right)$, we have
$$\widehat{G_{\gamma,j}}(\xi) = \widehat{f}\left(\xi-\frac{2\pi j}{h}\right)i^{\multi{\gamma}}\xi^\gamma\widetilde{\psi}\left(h\xi-2\pi j\right)\mah(\xi) =:\widehat{f}\left(\xi-\frac{2\pi j}{h}\right)\rho_{\gamma,j}(\xi).$$

By \eqref{EQmultiplieronenormbound} and
\eqref{EQmultiplierconvolutioninequality}, $\|G_{\gamma,j}\|_{L_p}$
is majorized by $\|\rho_{\gamma,j}^\vee\|_{L_1}\left\|\left(\widehat{f}\left(\cdot-\frac{2\pi
j}{h}\right)\right)^\vee\right\|_{L_p}$ by viewing $\rho_{\gamma,j}$ as an $L_p$ multiplier operator.  We now estimate $\|\rho_{\gamma,j}^\vee\|_{L_1}$.

Recall that if $\underset{\multi{\beta}\leq d+1}\sum\|D^\beta\rho_{\gamma,j}\|_{L_1}\leq M$, then $|\rho_{\gamma,j}^\vee(x)|\leq\frac{M}{1+|x|^{d+1}}$, thus $\|\rho_{\gamma,j}^\vee\|_{L_1}\leq CM$.

Since $\supp(\rho_{\gamma,j})\subset B_j:=B(\frac{2\pi j}{h},\frac{\pi+2\eps}{h})$, we have
$$\dint_{\R^d}\left|D^\beta\rho_{\gamma,j}(\xi)\right|d\xi\leq\left|\frac{\pi+2\eps}{h}\right|^d\left\|D^\beta\left(\xi^\gamma\widetilde{\psi}\left(h\xi-2\pi j\right)\mah(\xi)\right)\right\|_{L_\infty(B_j)}.$$
We may write the term at the right as the supremum of the absolute value of the following expression:
$$\underset{\beta_1+\beta_2+\beta_3 = \beta}\sum C_{\beta_1,\beta_2,\beta_3}D^{\beta_1}(\xi^\gamma)D^{\beta_2}\left(\widetilde{\psi}\left(h\xi-2\pi j\right)\right)D^{\beta_3}\mah(\xi).$$
Now we estimate each term in the series above. First, there is a constant $C$ depending on $\beta_1$ and $\gamma$ such that   $|D^{\beta_1}(\xi^\gamma)|\leq C|\xi|^{\multi{\gamma}-\multi{\beta_1}}\leq C\left|\frac{\pi(2|j|+1)+2\eps}{h}\right|^{\multi{\gamma}}.$  Since $\psi$ is $C^\infty,$ there is a constant $C$ depending only on $\multi{\beta_2}$ such that
$|D^{\beta_2}\widetilde{\psi}(h(\xi-\frac{2\pi j}{h}))|\leq Ch^{\multi{\beta_2}}.$  Finally, by Corollary \ref{Dm_est2}, there is a constant $C$ such that
$$\|D^{\beta_3}\mah(\xi)\|_{L_\infty(B_j)}\leq Ch^{\multi{\beta_3}}e^{-\frac{2\pi(|j|-1)}{3dh}}.$$
The appearance of $|j|-1$ in the exponential term rather than $|j|$ is due to the fact that the support of $\widetilde{\psi}$ overlaps multiple cubes, and thus an overlap into a cube corresponding to a smaller value of $|j|$ is possible.  We note that in some cases when $\|j\|_\infty=1$, this may cause an overlap of the support into cubes with $|j|=1$, but the final estimate above holds nonetheless by applying Corollary \ref{Dm_est3} with $\eps = \frac{2(|j|-1)}{3\sqrt{d}}$, which is indeed less than 1.  Note too that we need the additional assumption that the ball $B_j$ does not extend outside of any face of codimension larger than 1 of the cube $\left[-\frac{\pi(2j+1)+2\eps}{h},\frac{\pi(2j+1)+2\eps}{h}\right]^d$.  If this were not the case, then in $d=2$, for example, the ball corresponding to $j=(1,1)$ would have part of its support in the cube centered about the origin, which is undesirable since we have no uniform control on derivatives of $\mah$ on the boundary of this cube.  Similarly, in $d=3$, the ball corresponding to $j=(1,1,0)$ could overlap the line segment connecting $B_j$ to the cube at the origin.  However, if we assume that $\eps\leq \frac{(\sqrt{2}-1)\pi}{2}$, then such an overlap will not occur.  We note that the assumption on $\eps$ being small is not strictly necessary, but without it the estimates break into many more cases, which would unnecessarily convolute the proof.

Combining the estimates above provides the bound
$$\|\rho_{\gamma,j}^\vee\|_{L_1}\leq \left|\dfrac{\pi(2|j|+1)+2\eps}{h}\right|^{\multi{\gamma}+d}h^{d+1}e^{-\frac{2\pi(|j|-1)}{3dh}}.$$

{\em Case} $|j|=1$:  Since the end result is the same for each integer of length 1, we restrict our attentions to the case $j=e_1=(1,0,\dots,0)$ without loss of generality.  In this case, the ball $B_{e_1}$ containing the support of $\widehat{f}(\cdot-\frac{2\pi e_1}{h})$ overlaps one of the faces of the cube $[-\frac{\pi}{h},\frac{\pi}{h}]^d$, where we do not have uniform bounds (in $h$) for $\mah$.  Consequently, we cannot execute an argument similar to that in the case $|j|>1$.  However, let $\psi$, $\nu$, and $\omega$ be the smooth bump functions from Section \ref{SECstabilityunivariate}, and define $\widetilde{\psi}(\xi):=\psi(\xi_1)\dots\psi(\xi_d)$, $\nutilde(\xi)  :=\nu(\xi_1+\pi)\psi(\xi_2)\dots\psi(\xi_d)$, and $\wtilde(\xi):=\omega(\xi_1)\psi(\xi_2)\dots\psi(\xi_d)$. Then $\supp(\nutilde)=[-\pi-2\eps,-\pi+2\eps]\times[-\pi-2\eps,\pi+2\eps]^{d-1}$ and $\supp(\wtilde)=[-\pi+\eps,\pi+2\eps]\times[-\pi-2\eps,\pi+2\eps]^{d-1}$.  On the overlap of their supports, $\nutilde(\xi)+\wtilde(\xi)=1$, and for every $\xi$, $\nutilde(\xi)+\wtilde(\xi) = \widetilde{\psi}(\xi)$. Thus we may write
$$\widehat{G_{\gamma,e_1}}(\xi) = \widehat{f}\left(\xi-\frac{2\pi e_1}{h}\right)\left[\sigma_{1,\gamma,e_1}(\xi)+\sigma_{2,\gamma,e_1}(\xi)\right],$$
where
$$\sigma_{1,\gamma,e_1}(\xi) := i^{\multi{\gamma}}\xi^\gamma\nutilde\left(h\xi-2\pi e_1\right)\mah(\xi),$$
and $$\sigma_{2,\gamma,e_1}(\xi):=i^{\multi{\gamma}}\xi^\gamma\wtilde\left(h\xi-2\pi e_1\right)\mah(\xi).$$

Consider first that $\sigma_{1,\gamma,e_1}(\xi) = i^{\multi{\gamma}}\xi^\gamma\nu(h(\xi_1-\frac{\pi}{h}))\mah(\xi)$, since on the support of $\widehat{f}$, $\phi(h\xi_i)=1$ for $i=2,\dots,d$.
Let $\gamma' = \gamma-\gamma_1$, and identify $\xi_1^{\gamma_1}$ with its Taylor series to obtain

\begin{align*}
\sigma_{1,\gamma,e_1}(\xi) &= i^{\multi{\gamma}}\xi^{\gamma'}\finsum{l}{0}{\gamma_1}\binom{\gamma_1}{l}\left(\frac{2\pi}{h}\right)^{\gamma_1-l}\left(\xi_1-\frac{2\pi}{h}\right)^{l}\nu\left(h\xi_1-\pi\right)m_{\alpha,h}(\xi)\nonumber\\
& = i^{\multi{\gamma}}\xi^{\gamma'}\left(\xi_1-\frac{2\pi}{h}\right)^{\gamma_1}m_{\alpha,h}(\xi)\finsum{l}{0}{\gamma_1}\binom{\gamma_1}{l}\left(\frac{2\pi}{h}\right)^{\gamma_1-l}\frac{\nu(h\xi_1-\pi)}{(\xi_1-\frac{2\pi}{h})^{\gamma_1-l}}\nonumber\\
& =: i^{\multi{\gamma}}\left(\xi-\frac{2\pi e_1}{h}\right)^\gamma m_{\alpha,h}(\xi)\mu_{\gamma_1,e_1}(\xi).
\end{align*}

Consider $\mu_{\gamma_1,e_1}$ as a univariate Fourier multiplier, and note that
$$\|\mu_{\gamma_1,e_1}\|_{\Mp}\leq\finsum{l}{0}{\gamma_1}\binom{\gamma_1}{l}(2\pi)^{\gamma_1-l}\left\|\left(\dfrac{\nu(\cdot-\pi)}{(\cdot-2\pi)^{\gamma_1-l}}\right)^\vee\right\|_{L_1(\R)}\leq C,$$
where $C$ is a constant depending only on $\gamma_1$, which depends on $k$.  

Subsequently, notice that $\mu_{\gamma_1,e_1}$ may be viewed as a Fourier multiplier on $L_p(\R^d)$ via the tensor product representation $\mu_{\gamma_1,e_1}(\xi)= \mu_{\gamma_1,e_1}(\xi_1)I(\xi_2)\dots I(\xi_d)$, where $I$ is the constant 1 function, whose multiplier norm is 1.  Consequently $\|\mu_{\gamma_1,e_1}\|_{\Mp}$ is majorized by the constant $C$ above.

Moving on, we use an analogous estimate to the case $|j|>1$ to find a bound on the multiplier $\sigma_{2,\gamma,e_1}$.  All of the terms are estimated the same way, except that since $\wtilde$ is supported on $[\frac{(1+\eps)\pi}{h},\frac{3\pi}{h}]\times[-\frac{\pi}{h},\frac{\pi}{h}]^{d-1}$, we use Corollary \ref{Dm_est3} to find that the maximum of $|D^{\beta}\mah(\xi)|$ on the cube is a constant multiple of $h^{\multi{\beta}}e^{-\frac{\eps\pi}{\sqrt{d}h}}$.

Now the proof of Theorem \ref{THMstabilinterpolation} follows by the same calculation in Section \ref{SECstabilityunivariate}.

\section{Remark}\label{SECremark}

Since cardinal interpolation of smooth functions can be done using many of the classical RBFs, it is a natural question to ask for which RBFs can the results of \cite{hmnw} and this work be extended?  Cardinal interpolation using translates of thin plate splines, polyharmonic splines, and B-splines (see, for example, \cite{Buhmann,BuhmannBook}) has been considered.  However, the method of proof here makes essential use of the fact that the Fourier transform of the cardinal function decays exponentially away from the origin.  This is the case for the Gaussian and the general multiquadrics, but for each of the other RBFs mentioned, the Fourier transforms of their cardinal functions decay only algebraically away from the origin.  Therefore, if such results as are obtained here are to be extended to these other natural RBFs, it would seem that different techniques are needed.  

It is germane to mention the case of the Poisson kernel, which corresponds to $\alpha=-(d+1)/2$ in \eqref{EQgmcdef} above.  For this example, the Fourier transform of the cardinal function indeed decays exponentially; however, the argument in Section \ref{SECstabilitymultivariate} cannot be used since it requires the boundedness of the derivatives of $m_{\alpha,h}$ near the origin.  Thus the results cannot be extended to the multivariate Poisson kernel.  But when $d=1$, the derivatives near the origin are bounded, hence a straightforward modification of the proof given in Section \ref{SECstabilityunivariate} shows that the results of this paper may be extended to the univariate Poisson kernel, $(x^2+c^2)^{-1}$. In particular, the following holds.

\begin{theorem}
 Let $\alpha=-1$, $1<p<\infty$, $k\in\N$, and $0<h\leq1$.  There exists a constant $C$, independent of $h$, so that for every $g\in
W_p^k(\R)$,
 $$\|I_\alpha^hg-g\|_{L_p}\leq Ch^k\|g\|_{W_p^k}.$$
 
 If $p=1$ and $k>1$ or $p=\infty$ and $k\in\N$, there is a constant $C$, independent of $h$, so
that for every $g\in W_p^k(\R)$,
 $$\|I_\alpha^hg-g\|_{L_p}\leq C(1+|\ln h|)h^k\|g\|_{W_p^k}.$$
\end{theorem}

For general results on RBF interpolation of bandlimited functions at nonuniform sequences of points, refer to \cite{HammZonotopes,Ledford_Scattered,Ledford_Poisson,ss}, while \cite{Hamm} contains a nonuniform analogue to Theorem \ref{THMmaintheorem} for Gaussian interpolation in one dimension.  Also, a recent work of Buhmann and Dai \cite{BuhmannDai} examined pointwise error estimates for quasi-interpolation schemes involving RBFs.  For a pleasant discussion of the computational aspects and asymptotic behavior of interpolation with the Hardy multiquadric and other RBFs, the reader is referred to \cite{Madych}.


\appendix

\section{Proof of \eqref{EQIFourierDistribution}}
We now give the full proof of identity \eqref{EQIFourierDistribution} for completeness.  First, note that it suffices to prove the identity for the case when $h=1$, since the use of the Poisson summation formula will thus be justified for any dilation of the lattice.  Recall that a tempered distribution $T$ is $2\pi$--periodic provided $\bracket{\tau_{2\pi}T,\psi}=\bracket{T,\psi}$ for every $\psi\in\schwartz$.  For simplicity in writing the expressions that follow, all sums are to be taken over $\Z^d$.

Note that any periodic tempered distribution may be expressed as $\sum_{j}c_je_{-j}$, where the Fourier coefficients $c_j$ are of at most polynomial growth in $|j|$ .  In particular, if $T=\sum_{j}\tau_{2\pi j}\widehat{f}$, then $c_k=\bracket{\sum_{j}\tau_{2\pi j}(\widehat{f}e_{-k}),\psi},$ where $\psi$ is a $C^\infty$ bump function which forms a partition of unity for $\R^d$.  We have used here the fact that if $f\in E_{\pi+\eps}\cap W_p^k$, then $\sum_j\tau_{2\pi j}\widehat{f}$ is a $2\pi$--periodic distribution.

\begin{lemma}\label{LEMcj}
If $c_k$ are the Fourier coefficients of $T$ defined above, then $c_k=f(k)$.
\end{lemma}

\begin{proof}
Recalling that $\sum_j\tau_{2\pi j}\psi=1$, we have
\begin{align*}
    c_k &= \bracket{\sum_j\tau_{2\pi j}(\widehat{f}e_{-k}),\psi}\\
    & = \bracket{f,\left(e_{-k}\sum_{j}\tau_{2\pi j}\psi\right)^\wedge}\\
    & = \bracket{f,\tau_k\left(\sum_{j}\tau_{2\pi j}\psi\right)^\wedge}\\
    & = \bracket{f,\tau_k\delta}\\
    & = f(k).\\
\end{align*}
\end{proof}

Our first task is to define what $\sum_j\tau_{2\pi j}\widehat{f}m$ is as a tempered distribution (for ease, we drop the subscripts on $m$). To wit, set
\begin{equation}\label{EQTemp}\bracket{\sum_j\tau_{2\pi j}\widehat{f}m,\psi}:=\inflim{\lambda}\bracket{\sum_j\tau_{2\pi j}\widehat{f}m_\lambda,\psi},\end{equation}
where $m_\lambda$ is $m$ convolved with a standard $C^\infty$ approximate identity for $L_1$, with $\inflim{\lambda}m_\lambda=m$ in $L_1$.  First, note that the right-hand side of \eqref{EQTemp} is well-defined for any fixed $\lambda$ since $m_\lambda\in C^\infty$.  Moreover, for fixed $\lambda$, we may write $\bracket{\sum_j\tau_{2\pi j}\widehat{f}m_\lambda,\psi}=\bracket{\sum_jc_je_{-j}m_\lambda,\psi}$ from the discussion above, and consequently, if the limit exists, these limits are the same.

\begin{lemma}\label{LEMAppendix}
The following holds:
$$\bracket{\sum_jc_je_{-j}m,\psi}=\bracket{\sum_jc_j\tau_jm^\vee,\widehat\psi}.$$
\end{lemma}

\begin{proof}
To calculate the left-hand side, we need to find $\inflim{\lambda}\bracket{\sum_jc_je_{-j}m_\lambda,\widehat\psi}$.  For fixed $\lambda$, this can be expressed as an integral since $e_{-j}m_\lambda\in L_1$.  Consequently, we have
$$\bracket{\sum_jc_je_{-j}m_\lambda,\widehat\psi} = \sum_jc_j\dint m_\lambda(x)e^{-ijx}\dint\psi(\xi)e^{-ix\xi}d\xi dx.$$
An application of Fubini's theorem (which is justified by H\"{o}lder's inequality since $(c_j)=(f(j))\in\ell_p$, $m\in L_1$, and $\widehat\psi\in L_p$ for all $p$) yields that the right-hand side is equal to
$$\dint\sum_jc_jm_\lambda^\vee(\xi-j)\psi(\xi)d\xi.$$ Finally, taking the limit as $\lambda\to\infty$ is justified by the dominated convergence theorem.
\end{proof}

Now the proof of \eqref{EQIFourierDistribution} follows from Lemmas \ref{LEMcj} and \ref{LEMAppendix}:
\begin{align*}
\bracket{\left(\sum_jf(j)\tau_jm^\vee\right)^\wedge,\psi} & = \bracket{\sum_jf(j)\tau_jm^\vee,\widehat{\psi}}\\
& = \bracket{\sum_jc_je_{-j}m,\psi}\\
& = \bracket{\sum_j\tau_j\widehat{f}m,\psi}.   
\end{align*}

\section*{Acknowledgements}

The authors would like to thank the anonymous referees whose efforts greatly improved this article.

\end{document}